\documentclass[9pt]{article}
\usepackage[a4paper, total={7in, 10in}]{geometry}
\usepackage{amssymb, amsmath, amsthm, amstext, amscd}
\usepackage{bm}
\usepackage{mathrsfs}
\usepackage{latexsym, wasysym}
\usepackage{graphics}
\usepackage{verbatim}
\usepackage{amsgen, amsfonts, euscript, enumerate, enumitem, url, calc}
\usepackage{microtype}
\usepackage{colonequals} 
\usepackage{tikz}
\usetikzlibrary{cd}
\usetikzlibrary{decorations.pathmorphing}
\usepackage{stmaryrd}
\usepackage[all]{xy}
\usepackage{color}
\usepackage[colorlinks, pagebackref]{hyperref}
\hypersetup{
colorlinks=true,
citecolor=blue,
linkcolor=blue,
urlcolor=blue}

\usepackage{url} 
\makeatletter
\def\Url@twoslashes{\mathchar`\/\@ifnextchar/{\kern-.2em}{}}
\g@addto@macro\UrlSpecials{\do\/{\Url@twoslashes}}
\makeatother

\newtheorem{theorem}{Theorem}[section]
\newtheorem{lemma}[theorem]{Lemma}
\newtheorem{prop}[theorem]{Proposition}
\newtheorem{corollary}[theorem]{Corollary}

\newtheorem{sub-lemma}[theorem]{Sub-lemma}
\newtheorem{theorem-definition}[theorem]{Theorem-Definition}

\theoremstyle{definition}

\newtheorem{remark}[theorem]{Remark}

\theoremstyle{theorem}
\newtheorem{a-theorem}{A-Theorem}[section]
\newtheorem{a-lemma}[a-theorem]{A-Lemma}
\newtheorem{a-prop}[a-theorem]{A-Proposition}
\newtheorem{a-corollary}{A-Corollary}[section]
\theoremstyle{definition}
\newtheorem{a-definition}[a-theorem]{A-Definition}
\newtheorem{a-example}[a-theorem]{A-Example}
\newtheorem{a-examples}[a-theorem]{A-Examples}
\newtheorem{a-remark}[a-theorem]{A-Remark}

\theoremstyle{theorem}
\newtheorem*{theorem*}{Theorem}
\newtheorem*{question*}{Question}

\newcommand{\suchthat}{\;\ifnum\currentgrouptype=16 \middle\fi|\;}

\DeclareMathOperator*{\Spec}{\mathrm{Spec}}

\DeclareMathOperator*{\rank}{\mathrm{rk}}

\title{Comparison of the two notions of characteristic cycles}
\date{}
\author{Ankit Rai}

\AtEndDocument{\bigskip{\footnotesize%
\textsc{Chennai Mathematical Institute, Chennai, India, 603103}\par
\textit{E-mail address} : \texttt{ankitr@cmi.ac.in}}}

\begin{document}
\maketitle

\begin{abstract}
\noindent
Given a constructible sheaf $F$ on a complex manifold, Kashiwara-Schapira defined the notion of singular support and characteristic cycle of $F$. On the other hand for a Zariski constructible \'{e}tale sheaf $F$ on an algebraic variety $X$, Beilinson defined the notion of singular support of $F$ and Saito defined the notion of characteristic cycle of $F$.
In this article we compare these notions and prove that they agree in a suitable sense. In the appendix we discuss extension of the notions of singular support and characteristic cycles developed by Umezaki-Yang-Zhao and Barrett.
\end{abstract}

\tableofcontents

\section{Introduction} \label{sec:intro}
The comparison theorems between the classical and the modern geometry are of interest for multiple reasons.
On one hand it allows the use of powerful algebro-geometric methods to prove theorems in classical geometry, and on the other hand it provides a much needed testing ground for the intuitions behind developing new notions in modern geometry.\\

\noindent
In the next paragraph we discuss the necessary notation to state the main theorem of this article. Let $X$ be a separated smooth scheme of finite type over the field of complex numbers $\mathbb{C}$ and $X^{an}$ denotes the associated complex analytic space $X(\mathbb{C})$. Let $\Lambda$ be a finite local ring and $\mathrm{D}^b_{ctf}(X^{an}, \Lambda)$ (respectively $\mathrm{D}^b_{ctf}(X_{\text{\'{e}t}}, \Lambda)$)  denote the bounded derived category of tor-finite complexes $F$ of sheaves (respectively the \'{e}tale sheaves) with coefficients in finite $\Lambda$-modules, such that the cohomology sheaves $\mathrm{H}^i(F)$ are Zariski constructible and are of finite tor-dimension.
Kashiwara-Schapira refined\footnote{Let $\pi : T^*X \rightarrow X$ be the cotangent bundle then $\pi(\mathrm{SS^{KS}}(F)) = $ the support of $F \in \mathrm{D}^b_{ctf}(X, \Lambda)$ and this justifies the claim that singular support refines the notion of the support of a sheaf.} the notion of the support of $F \in \mathrm{D}^b_{ctf}(X^{an}, \Lambda)$ to that of the singular support.
The singular support of $F \in \mathrm{D}^b_{ctf}(X^{an}, \Lambda)$ is a closed complex analytic conical subset of $(T^*X)(\mathbb{C})$ which we denote by $\mathrm{SS^{KS}}(F)$.
In \cite{Beilinson-constructible}, Beilinson defines the notion of the singular support of $F \in \mathrm{D}^b_{ctf}(X_{\text{\'{e}t}}, \Lambda)$ as a Zariski closed conical subset of $T^*X$ denoted in this article by $\mathrm{SS^B}(F)$.
The notion of the singular support in the sense of Beilinson has been extended to any $F \in \mathrm{D}^b_{c}(X_{\text{\'{e}t}}, \Lambda)$ for $\Lambda \in \lbrace \mathbb{Z}_{\ell}, \mathbb{Q}_{\ell} \rbrace$ (See \cite[\S 5.15, Thm. 5.17]{Umezaki-Yang-Zhao-20}). We denote this by $\mathrm{SS^{BUYZ}}(F)$. In this article, we prove the following

\begin{theorem} \label{thm:intro-SS}
Let $F \in \mathrm{D}^b_c(X_{\textup{\'{e}t}}, \Lambda)$, where $\Lambda \in \lbrace \mathbb{Z}/\ell^n\mathbb{Z}, \mathbb{Z}_{\ell}, \mathbb{Q}_{\ell} \rbrace$ and let $\overline{\mathrm{SS^{KS}}(F)}^{Zar}$ denote the Zariski closure of $\mathrm{SS^{KS}}(F)$ with the reduced induced subscheme structure. Then $\overline{\mathrm{SS^{KS}}(F)}^{Zar} = \mathrm{SS^{BUYZ}}(F)$.
\end{theorem}

\noindent
The notion of the singular support can be upgraded to a cycle supported on the singular support. This has been achieved by Kashiwara-Schapira in the analytic setting, and by Saito in the algebraic setting.
The characteristic cycle defined by Kashiwara-Schapira is defined only for $F \in \mathrm{D}^b_c(X^{an}, \Lambda)$, where $\Lambda$ is a field of characteristic 0. We denote this cycle by $\mathrm{CC^{KS}}(F)$.
The characteristic cycle defined by Saito, a priori, only makes sense for $F \in \mathrm{D}^b_{ctf}(X_{\text{\'{e}t}}, \Lambda)$, where $\Lambda$ is a finite local ring. This notion has also been extended to any $F \in \mathrm{D}^b_c(X_{\text{\'{e}t}}, \Lambda)$ for $\Lambda \in \lbrace \mathbb{Z}_{\ell}, \mathbb{Q}_{\ell} \rbrace$ (See \cite[\S 5.15]{Umezaki-Yang-Zhao-20}). We denote it by $\mathrm{CC^{SUYZ}}(F)$.
The two notions are compared in the following

\begin{theorem} \label{thm:intro-CC}
For any $F \in \mathrm{D}^b_c(X_{\text{\'{e}t}}, \mathbb{Q}_{\ell})$, suppose $\mathrm{CC^{KS}}(F) = \underset{i}{\sum} m_i [X_i]$ as a Lagrangian cycle (see equation \eqref{eqn:lagrangian-cycle-cohomology-class}), where $X_i$ are irreducible components of $\mathrm{SS^{KS}}(F)$ then $\mathrm{CC^S}(F) = \underset{i}{\sum} m_i [\overline{X_i}^{Zar}]$ as cycles.
\end{theorem}

\subsection*{Organization of the paper and the strategy of the proof} \label{subsec:organization-proof-strategy}
In \S \ref{sec:basics} we lay down the assumptions made in this article and discuss some basic terminology which will be needed in the later sections.
In \S \ref{subsec:analytic} we recall the definition of the singular support and the characteristic cycle of a complex of sheaves on a complex algebraic variety. We have taken the liberty of stating a theorem of Kashiwara-Schapira as the definition of the singular support.
Section \ref{subsec:algebraic} recalls the definition of the weak singular support given by Beilinson and the definition of the characteristic cycle given by Saito. The extension of these notions to complexes of sheaves with coefficients in $\mathbb{Z}_{\ell}$ or $\mathbb{Q}_{\ell}$ is also discussed there.\\

\noindent
A few (perhaps well known) statements whose proofs we were unable to find in the literature are needed in the sequel. The statements along with their proofs are stated as Lemma \ref{lemma:char-cycle-shriek}, \ref{lemma:char-cycle-equal-pushforward} and \ref{lemma:irreducible-jordan-quotient} in \S \ref{subsec:lemmas}.
In \S \ref{subsec:summary} we summarise various properties of the singular supports and of the characteristic cycles that are to be used in the proofs of Theorems \ref{thm:intro-SS} and \ref{thm:intro-CC}. The proofs of Theorem \ref{thm:intro-SS} and Theorem \ref{thm:intro-CC} are completed in \S \ref{sec:main-theorem-proof}.
The strategy of the proof of these theorems is to recursively use properties 1 and 2 from \S\ref{subsec:summary} in order to reduce to the case of irreducible perverse sheaves. Further Hironaka's resolution of singularities and the decomposition theorem with respect to support for perverse sheaves (Step 3 of proof of Theorem \ref{thm:main}) are used to reduce to the case of irreducible perverse sheaves of the form $j_!L$, where $L$ is a simple local system on $U$, and $j : U \hookrightarrow X$ is an open subset whose complement is a strict normal crossing divisor.
The theorem is achieved thereafter by an explicit computation (See \S \ref{sec:main-theorem-proof}, Step 4).\\

\noindent
Finally Appendix A we discusses the dependence of the characteristic cycle $\mathrm{CC^{KS}}$ on the coefficient system and Appendix B compares the two ways of extending of the notion of singular support developed by Umezaki-Yang-Zhao and Barrett.

\subsection*{Acknowledgements}
The author thanks K. V. Shuddhodan for suggesting the question and for numerous discussions. Thanks are due to Najmuddin Fakhruddin for directing the author to relevant reference in Fulton and for the accompanied explanation which forms the proof of Lemma \ref{lemma:cycle-class-functorial}.

\section{Basics} \label{sec:basics}
\subsection{Assumptions in this article}
In this article an algebraic variety will mean a separated smooth scheme of finite type over the field $\mathbb{C}$ of complex numbers. Any morphism of algebraic varieties that appear will automatically be of finite type.
By $X^{an}$ we will mean the complex points of an algebraic variety, and any morphism will be assumed to be one induced by a morphism of algebraic varieties.
To keep matters simple and the exposition lucid we restrict ourselves to the case of $\Lambda \in \lbrace \mathbb{Z}/\ell^n\mathbb{Z}, \mathbb{Z}_{\ell}, \mathbb{Q}_{\ell} \rbrace$.

\subsection{Constructible sheaves}
Notion of a constructible sheaf depends on the stratification of a space and the fundamental groups of these strata. Since the latter two notions differ in the settings of algebraic varieties and complex analytic varieties, the notion of a constructible sheaf differs as well.
We recall the basics of constructible sheaves which also serves the purpose of itroducing the relevant notation.\\

\noindent
Let $\Lambda$ be a noetherian ring. A sheaf $F$ valued in finitely generated $\Lambda$-modules on $X^{an}$ is called \textit{constructible} if there exists finitely many Zariski locally closed subsets $X^{an}_i \subset X^{an}$ such that, (1) $X^{an} = \underset{i}{\sqcup}X^{an}_i$, (2) the Zariski closure $\overline{X^{an}_i}$ of each $X^{an}_i$ is a union of finitely many of the $X^{an}_j$, and (3) $F|_{X^{an}_i}$ is a local system valued in finite $\Lambda$-modules. We can furthermore choose $X^{an}_i$ to be complex submanifolds of $X^{an}$(See \cite[Prop. 8.5.4]{Kashiwara-Schapira}).
Let $\mathrm{D}^b(X^{an}, \Lambda)$ be the bounded derived category of sheaves on $X^{an}$ valued in $\Lambda$-modules. The full subcategory $\mathrm{D}^b_{ctf}(X^{an}, \Lambda) \subset \mathrm{D}^b(X^{an}, \Lambda)$ is then defined by declaring that $F \in \mathrm{Ob}(\mathrm{D}^b(X^{an}, \Lambda))$ belongs to $\mathrm{D}^b_{ctf}(X^{an}, \Lambda)$ if $\mathrm{H}^i(F)$ is a constructible sheaf, and $\mathrm{H}^q(F \otimes^LQ) \neq 0$ for finitely many $q$ and any finite $\Lambda$-module $Q$.
The assumptions on the ring $\Lambda$ ensure that the conditions on $F$ in order to belong to the subcategory $\mathrm{D}^b_{ctf}(X^{an}, \Lambda)$ is equivalent to the condition that $F_m$, the stalk of $F$ at any $m \in M$, is a perfect complex of $\Lambda$-modules (thanks to the standard results \cite[Lemma 066E, Lemma 0658]{stacks-project}).
We will denote $\mathrm{D}^b_{ctf}(X^{an}, \Lambda)$ by $\mathrm{D}^b_c(X^{an}, \Lambda)$ henceforth.\\

\noindent
Let $\Lambda$ be a finite local ring. Let $F$ be a Zariski constructible \'{e}tale sheaf valued in finite $\Lambda$-modules on the algebraic variety $X$ that is $F$ is an \'{e}tale sheaf valued in finitely generated $\Lambda$-modules and there exists finitely many locally closed subvarieties $X_i \subset X$ such that (1) $X = \underset{i}{\sqcup}X_i$, (2) the Zariski closure $\overline{X_i}$ of each $X_i$ is a union of finitely many of the $X_j$, and (3) $F|_{X_i}$ is a \'{e}tale-locally constant sheaf valued in finite $\Lambda$-modules.
The Zariski locally closed subvarieties $X_i$ can always be chosen to be smooth. Let $\mathrm{D}^b(X_{\text{\'{e}t}}, \Lambda)$ denote the bounded derived category of \'{e}tale sheaves valued in finite $\Lambda$-modules.
The full subcategory $\mathrm{D}^b_{ctf}(X_{\text{\'{e}t}}, \Lambda) \subset \mathrm{D}^b(X_{\text{\'{e}t}}, \Lambda)$ is defined by declaring that $F \in \mathrm{Ob}(\mathrm{D}^b(X_{\text{\'{e}t}}, \Lambda))$ belongs to $\mathrm{D}^b_{ctf}(X_{\text{\'{e}t}}, \Lambda)$ if $\mathrm{H}^i(F)$ is a Zariski constructible \'{e}tale $\Lambda$-sheaf which is nonzero for atmost finitely many $i$, and $F \otimes^L_{\Lambda} Q \in \mathrm{D}^b(X_{\text{\'{e}t}}, \Lambda)$ for any finitely generated module $Q$ over $\Lambda$.
In the sequel we will denote $\mathrm{D}^b_{ctf}(X_{\text{\'{e}t}}, \Lambda)$ by $\mathrm{D}^b_c(X_{\text{\'{e}t}}, \Lambda)$.\\

\noindent
For $\Lambda = \mathbb{Z}_{\ell}$, the derived category $\mathrm{D}^b_c(X_{\text{\'{e}t}}, \mathbb{Z}_{\ell})$ is defined to be the category $2$-$\lim\mathrm{D}^b_c(X_{\text{\'{e}t}}, \Lambda/\ell^n)$ whose objects are projective systems $F = \lbrace F_n \rbrace_{n\geq 1}$ with $F_n \in \mathrm{D}^b_c(X_{\text{\'{e}t}}, \mathbb{Z}/\ell^n\mathbb{Z})$ such the induced map $F_{n+1}\otimes^L\mathbb{Z}/\ell^n\mathbb{Z} \cong F_n$ is an isomorphism, and the morphisms $f \in \hom_{\mathrm{D}^b_c(X_{\text{\'{e}t}}, \mathbb{Z}_{\ell})}(F, G)$ is a collection $\lbrace f_n \rbrace$ of morphisms, where $f_n \in \hom_{\mathrm{D}^b_c(X_{\text{\'{e}t}}, \mathbb{Z}/\ell^n\mathbb{Z})}(F_n, G_n)$ which renders the following diagram
\[
\begin{tikzcd}[column sep = 20ex]
	F_{n+1}\otimes^L\mathbb{Z}/\ell^n\mathbb{Z} \arrow[d, "\sim" {anchor=south, rotate=90}] \arrow[r, "f_{n+1}\otimes \mathbb{Z}/\ell^n\mathbb{Z}"] & G_{n+1}\otimes^L\mathbb{Z}/\ell^n\mathbb{Z} \arrow[d, "\sim" {anchor=north, rotate=90}]\\
	F_n \arrow[r, "f_n"'] & G_n
\end{tikzcd}
\]
to be commutative.
The $\ell$-adic derived category $\mathrm{D}^b_c(X_{\text{\'{e}t}}, \mathbb{Q}_{\ell})$ is obtained by inverting the multiplication by $\ell$ map on $\mathrm{D}^b_c(X_{\text{\'{e}t}}, \mathbb{Z}_{\ell})$.
The definition of $\mathrm{D}^b_c(X^{an}, \Lambda)$ makes sense without any further qualifications. Over the algebraically closed field $\mathbb{C}$, any \'{e}tale-locally constant sheaf on an algebraic variety $X$ over $\mathbb{C}$ valued in a finite ring $\Lambda$ gives a local system on $X^{an}$.
This allows us to associate to $F \in \mathrm{D}^b_{c}(X_{\text{\'{e}t}}, \mathbb{Z}_{\ell})$ $\left(\mathrm{resp.\,} F \in \mathrm{D}^b_{c}(X_{\text{\'{e}t}},\mathbb{Z}_{\ell})[\ell^{-1}]\right)$ an element in $\mathrm{D}^b_c(X^{an}, \mathbb{Z}_{\ell})$ (in $\mathrm{D}^b_{c}(X^{an}, \mathbb{Q}_{\ell})$) in a canonical manner. The associated object is denoted by $\varepsilon^*F$ in \S \ref{subsec:comparison}.

\subsection{Six functor formalism} \label{subsec:six-functor-formalism}
Let $X$ and $Y$ be algebraic varieties and $f : X \rightarrow Y$ be a morphism of algebraic varieties. Let $\Lambda \in \lbrace \mathbb{Z}/\ell^n\mathbb{Z}, \mathbb{Z}_{\ell}, \mathbb{Q}_{\ell} \rbrace$ we may associate to a morphism $f : X \rightarrow Y$ two pairs of adjoint functors
\[
(f^*,f_*), (f_!, f^!) : \begin{tikzcd} \mathrm{D}^b_c(Y_{\text{\'{e}t}}, \Lambda) \arrow[r, yshift=-.5ex] & \mathrm{D}^b_c(X_{\text{\'{e}t}}, \Lambda) \arrow[l, yshift=.5ex] &\text{and,}& (f^*,f_*), (f_!, f^!) : \mathrm{D}^b_c(Y^{an}, \Lambda) \arrow[r, yshift=-.5ex] & \mathrm{D}^b_c(X^{an}, \Lambda) \arrow[l, yshift=.5ex] \end{tikzcd}
\]
The adjointness of the pair $(f_!,f^!)$ is called as the Verdier duality which is a vast generalization of the Poincar\'{e} duality. The existence of the pair of adjoint functors maybe found in \cite[Cor. 2.2.2, Cor. 2.2.5]{Schurman-constructible} and \cite[Cor. 1.5]{Deligne-SGA4.5-finiteness}.\\

\noindent
Consider the triple $\begin{tikzcd} U \arrow[r, hook, "j"] & X & \arrow[l, hook', "i"'] Z \end{tikzcd}$, where $j$ is an open immersion and $i$ is a closed immersion. It is easy to check from the definitions that $i_*=i_!$, $j^*=j^!$ and $j^*i_* = 0$. Moreover, $i_*,j_*$, and $j_!$ are fully faithful.
For a sheaf $F \in \mathrm{D}^b_c(X_{\text{\'{e}t}}, \Lambda)$ (resp. $\mathrm{D}^b_c(X^{an}, \Lambda)$) we have the following distinguished triangles
\[
	\rightarrow j_!j^*F \rightarrow F \rightarrow i_*i^*F \xrightarrow{+1} \qquad\text{and,}\qquad \rightarrow i_*i^!F \rightarrow F \rightarrow j_*j^*F \xrightarrow{+1}
\]
in $\mathrm{D}^b_c(X_{\text{\'{e}t}}, \Lambda)$ (resp. $\mathrm{D}^b_c(X^{an}, \Lambda)$). The data described in this paragraph form a recollement of $\mathrm{D}^b_c(X_{\text{\'{e}t}}, \Lambda)$ (resp. $\mathrm{D}^b_c(X^{an}, \Lambda)$) in the sense of \cite[\S 1.4.3]{BBD}.

\subsection{Comparison of analytic and \'{e}tale topos} \label{subsec:comparison}
Recall from \cite[\S 6.1, 6.2]{BBD} that there is a morphism of topos $X^{an} \rightarrow X_{\text{\'{e}t}}$ which induces a fully faithful functor
\[
\begin{tikzcd}[row sep = .2ex]
	\varepsilon^* : \lbrace \text{Zariski constructible \'{e}tale } \Lambda\text{-sheaf} \rbrace \arrow[r] & \lbrace \text{Zariski constructible } \Lambda\text{-sheaf} \rbrace\\
	\varepsilon^* : \mathrm{D}^b_c(X_{\text{\'{e}t}}, \mathbb{Q}_{\ell}) \arrow[r] &\mathrm{D}^b_c(X^{an}, \mathbb{Q}_{\ell}).
\end{tikzcd}
\]
The essential image of the functor $\varepsilon^*$ consists of objects $F$ such that $\mathrm{H}^iF$ is the image of a Zariski constructible \'{e}tale $\Lambda$-sheaf under $\varepsilon^*$.
In fact, as noted in \cite{BBD} the functor is not an equivalence of categories.

\subsection{Vanishing cycles} \label{subsec:vanishing-cycles}
Let $X^{an}$ be a complex algebraic variety, $F \in \mathrm{D}^b_c(X^{an}, \Lambda)$ with $\Lambda \in \lbrace \mathbb{Z}/\ell^n\mathbb{Z}, \mathbb{Z}_{\ell}, \mathbb{Q}_{\ell} \rbrace$, and $f : X^{an} \rightarrow \mathbb{C}$ be a morphism.
Set $X^{an}_0 = f^{-1}(0)$ and denote by $i$ the closed embedding $i : X^{an}_0 \hookrightarrow X^{an}$. Let $\widetilde{p} : \widetilde{X} \rightarrow X^{an}$ be the pullback of $f$ along the composite $\widetilde{\mathbb{C}^{\times}} \rightarrow \mathbb{C}^{\times} \hookrightarrow \mathbb{C}$, where $\widetilde{\mathbb{C}^{\times}}$ is the universal cover of $\mathbb{C}^{\times}$.
The vanishing cycle $\phi^{an}_f(F)$ of $F$ with respect to the map $f$ is defined to be the cone $\mathrm{Cone}(i^*F \rightarrow i^*\widetilde{p}_*\widetilde{p}^*F)$. The superscript ${}^{an}$ has been added to the standard notation to remind us that the objects $f, X^{an}$, and $F$ are analytic in nature. The functor $\phi^{an}_f : \mathrm{D}^b(X^{an}, \Lambda) \rightarrow \mathrm{D}^b(X^{an}_0, \Lambda)$ preserves constructibility (\cite[Thm. 4.0.2]{Schurman-constructible}).\\

\noindent
Let $X$ be an algebraic variety and $f: X \rightarrow \mathbb{A}^1$ be a morphism of algebraic varieties. Let $\mathcal{O}^{sh}_{\mathbb{A}^1, \lbrace 0 \rbrace}$ denote the strict henselization of the local ring of $\mathbb{A}^1$ at the point $\lbrace 0 \rbrace$.
We continue to denote by $f$ the base change of $f$ via the map $\Spec(\mathcal{O}^{sh}_{\mathbb{A}^1, \lbrace 0 \rbrace}) \rightarrow \mathbb{A}^1$.
Let $\bar{s}$ and $\bar{\eta}$ respectively be a closed geometric and a generic geometric point of $\Spec(\mathcal{O}^{sh}_{\mathbb{A}^1, \lbrace 0 \rbrace})$. Let $i : X_{\bar{s}} \rightarrow X$ and $j : X_{\bar{\eta}} \rightarrow X$ respectively be the pullback of $\bar{s} \rightarrow \Spec(\mathcal{O}^{sh}_{\mathbb{A}^1, \lbrace 0 \rbrace})$ and $\bar{\eta} \to \Spec(\mathcal{O}^{sh}_{\mathbb{A}^1, \lbrace 0 \rbrace})$ along the morphism $f$.
The natural map $F \rightarrow j_*j^*F$ induces the map $i^*F \rightarrow i^*j_*j^*F$. The vanishing cycle $\phi^{alg}_f(F)$ of $F$ with respect to the map $f$ is defined to be the cone $\mathrm{Cone}(i^*F \rightarrow i^*j_*j^*F)$.
The superscript ${}^{alg}$ has been added to the standard notation to remind ourselves that the objects $f, X,$ and $F$ are algebraic in nature. The functor $\phi^{alg}_f : \mathrm{D}^b(X, \Lambda) \rightarrow \mathrm{D}^b(X_0, \Lambda)$ preserves constructibility.
This is given by \cite[Thm. 3.2]{Deligne-SGA4.5-finiteness} when $\Lambda = \mathbb{Z}/\ell^n \mathbb{Z}$ for some $n$. The case of $\Lambda = \mathbb{Z}_{\ell}$ follows from applying the aforementioned theorem for $\Lambda = \mathbb{Z}/\ell^n\mathbb{Z}$ for all $n \geq 1$ by the standard arguments, and finally the case of $\Lambda = \mathbb{Q}_{\ell}$ is immediate.\\

\noindent
When $X$ is an algebraic variety and $f : X \rightarrow \mathbb{A}^1$ is a morphism of algebraic varieties, we have two notions of vanishing cycles $\phi^{alg}_f(F)$ and $\phi^{an}_{f}(\varepsilon^*F)$. The morphism of topos $\varepsilon$ gives the canonical map
\[
	\mathrm{comp}_{\text{\'{e}t, Betti}} : \varepsilon^*(\phi^{alg}_f(F)) \rightarrow \phi^{an}_{f(\mathbb{C})}(\varepsilon^*(F)).
\]
The following comparison result for vanishing cycles is due to Deligne.
\begin{theorem}[\cite{Deligne-Katz}, Expose XIV, Theorem 2.8] \label{thm:comp-vanishing-cycle}
The map $\mathrm{comp}_{\textup{\'{e}t, Betti}}$ is an isomorphism.
\end{theorem}

\noindent
Let $X$ and $S$ be algebraic varieties, $f : X \rightarrow S$ be a morphism, and $F$ be in $ \mathrm{D}^b_{ctf}(X_{\text{\'{e}t}}, \Lambda)$ with $\Lambda$ a finite local ring. Let $\bar{x} \in X(\mathbb{C})$ be a geometric point of $X$ and $f(\bar{x}) \in S(\mathbb{C})$ be the image of $\bar{x}$ under $f$.
The map $f$ induces a local homomorphism of local rings $\mathcal{O}_{S, f(\bar{x})} \rightarrow \mathcal{O}_{X, \bar{x}}$, and hence also a homomorphism on their strict henselizations $\mathcal{O}^{sh}_{S, f(\bar{x})} \rightarrow \mathcal{O}^{sh}_{X, \bar{x}}$. Let $f^{sh} : \Spec(\mathcal{O}^{sh}_{X, \bar{x}}) \rightarrow \Spec(\mathcal{O}^{sh}_{S, f(\bar{x})})$ be the induced map.
Let $M_{\bar{x}, \bar{s}}$ denote the fiber product $\Spec(\mathcal{O}^{sh}_{X,\bar{x}}) \times_{\Spec(\mathcal{O}^{sh}_{S, f(\bar{x})})} \bar{s}$ for a geometric point $\bar{s} \in \Spec(\mathcal{O}^{sh}_{S, f(\bar{x})})$. The canonical map $M_{\bar{x}, \bar{s}} \rightarrow \Spec(\mathcal{O}^{sh}_{X, \bar{x}})$ induces a pullback map
\[
	\alpha_{\bar{x}, \bar{s}} : \Gamma(\Spec(\mathcal{O}^{sh}_{X, \bar{x}}), F) \rightarrow \Gamma(M_{\bar{x}, \bar{s}}, F).
\]
We say that $F$ is locally acyclic with respect to $f$ if $\alpha_{\bar{x}, \bar{s}}$ are isomorphisms for all geometric points $\bar{x} \in X$ and a generization $\bar{s}$ of $f(\bar{x}) \in S$. In the situation when the algebraic variety $S$ is of dimension 1, we have the following well known lemma which relates the local acyclicity with the vanishing cycles.

\begin{lemma}
Suppose that $X, S, f : X \rightarrow S$, and $F$ are as above. Then $F$ is locally acyclic with respect to the morphism $f \Leftrightarrow \phi_f(F) = 0$.
\end{lemma}
\begin{proof}
$\phi_f(F) = 0 \Leftrightarrow i^*F \rightarrow i^*j_*j^*F$ is an isomorphism $\Leftrightarrow F \rightarrow j_*j^*F$ is an isomorphism. Taking stalks at a geometric point $\bar{x} \in X_{\bar{s}}$ we get
\[
	\Gamma(\Spec(\mathcal{O}^{sh}_{X, \bar{x}}, F) \cong (j_*j^*F)_{\bar{x}} \cong \Gamma(\Spec(\mathcal{O}^{sh}_{X, \bar{x}}), j_*j^*F) \cong \Gamma(\Spec(\mathcal{O}^{sh}_{X, \bar{x}})_{\bar{\eta}}, F).
\]
The isomorphism in the last step follows due to the property that nearby cycles commutes with a finite base change (See \cite[Prop. 2.7]{Saito}). The isomorphism of the two extreme terms precisely means that $F$ is locally acyclic with respect to the morphism $f$. The lemma follows.
\end{proof}

\noindent
In a more general setup than the one considered above, the given definition of local acyclicity has several issues such as, the sheaf of vanishing cycles may not be constructible\footnote{But it is constructible after a modification of the base (See \cite[Thm. 6.1]{Orgogozo-nearby-cycles})}, the vanishing cycles functor may not commute with the base change, and the data of the local acyclicity cannot be captured by a lemma as simple as above.
The appearance of non-finite type schemes such as the Milnor tubes can be circumvented by using the characterization of universally acyclic complex of sheaves developed by Lu-Zheng \cite{Lu-Zheng} and Hansen-Scholze \cite{Hansen-Scholze}.

\subsection{Cycle class} \label{subsec:cycle-class-functorial}
Let $X$ be an equidimensional algebraic variety. For a natural number $n$, let $\mathrm{CH}^n(X)$ denote the formal sum of irreducible algebraic varieties of $X$ of codimension $n$ upto rational equivalence. There is the following cycle class map
\[
	\mathrm{CH}^n(X) \xrightarrow{\text{cl}} \mathrm{H}^{2n}(X_{\text{\'{e}t}}, \mathbb{Q}_{\ell}),
\]
assigning to a closed subvariety $Y$ of $X$ of codimension $n$, a refined cycle class $[Y] \in \mathrm{H}^{2n}_{Y}(X_{\text{\'{e}t}}, \mathbb{Q}_{\ell})$ (See \cite[Ch. III, Def. 2, Remark 2]{ElZein}) which under the canonical map
$\mathrm{H}^{2n}_{Y}(X_{\text{\'{e}t}}, \mathbb{Q}_{\ell}) \rightarrow \mathrm{H}^{2n}(X_{\text{\'{e}t}}, \mathbb{Q}_{\ell})$ maps to $\text{cl}(Y)$. We will write $\mathrm{cl}$ for the refined cycle class map as well. Let $c$ be a correspondence as below
\[
\begin{tikzcd}
& C \arrow[dl, "f"'] \arrow[dr, "g"]&\\
X && Z
\end{tikzcd},
\]
where $C$ is an algebraic variety, $f$ is a locally complete intersection morphism and $g$ is proper. Then, $c_*$ the pushforward map along $c$ is defined to be the composite $g_*f^*$,
\[
	c_* : \mathrm{CH}^n(X) \rightarrow \mathrm{CH}^{n+\dim Z -\dim C}(Z).
\]
Similarly\footnote{Here we do not need the assumption of $f$ beiung lci.}, the maps
\[
	c_* : \mathrm{H}^{2n}(X_{\text{\'{e}t}}, \mathbb{Q}_{\ell}) \rightarrow \mathrm{H}^{2n+2\dim Z-2\dim C}(Z_{\text{\'{e}t}}, \mathbb{Q}_{\ell}) \qquad \text{ and } \qquad c_* : \mathrm{H}^{2n}_{Y}(X_{\text{\'{e}t}}, \mathbb{Q}_{\ell}) \rightarrow \mathrm{H}^{2n+2\dim Z-2\dim C}_{g\circ f^{-1}(Y)}(Z_{\text{\'{e}t}}, \mathbb{Q}_{\ell})
\]
is defined to be the composite $g_*f^*$.\\

\noindent
Recall that by our assumption the varieties $X, Z$, and $C$ are smooth. This in particular implies that the morphisms are locally complete intersections\footnote{abrreviated as lci in the rest of the document.}. The cycle class map is functorial with respect to the pushforward maps under correspondences, that is, the diagram below commutes.
\[
\begin{tikzcd}
\mathrm{CH}^n(X) \arrow[r, "\text{cl}"] \arrow[d, "c_*"] & \mathrm{H}^{2n}(X_{\text{\'{e}t}}, \mathbb{Q}_{\ell}) \arrow[d, "c_*"]\\
\mathrm{CH}^{n+\dim Z-\dim C}(Z) \arrow[r, "\text{cl}"] &\mathrm{H}^{2n+2\dim Z-2\dim C}(Z_{\text{\'{e}t}}, \mathbb{Q}_{\ell}).
\end{tikzcd}
\]
\begin{lemma} \label{lemma:cycle-class-functorial}
Let $c$ be a correspondence as above,  $Y \subset X$ be an equidimensional Zariski closed subset with the reduced induced subscheme structure, let $c^0(Y)$ denote the Zariski closed subset $g\circ f^{-1}(Y)$, and let $n$ be the codimension of $Y$ in $X$. Then the refined cycle class map is also functorial. That is, the diagram below commutes.
\[
\begin{tikzcd}
\mathrm{CH}^0(Y) \arrow[r, "\mathrm{cl}"] \arrow[d, "c_*"] & \mathrm{H}^{2n}_{Y}(X_{\text{\'{e}t}}, \mathbb{Q}_{\ell}) \arrow[d, "c_*"]\\
\mathrm{CH}^{\dim Z-\dim C}(c^0(Y)) \arrow[r, "\mathrm{cl}"] &\mathrm{H}^{2n+2\dim Z-2\dim C}_{c^0(Y)}(Z_{\text{\'{e}t}}, \mathbb{Q}_{\ell})
\end{tikzcd}
\]
\end{lemma}
\begin{proof}
The commutativity of the above diagram can be broken into functoriality of the refined cycle class map for 1. pullback along locally complete intersection morphism\footnote{Note that any morphism among smooth varieties is a locally complete intersection.} and 2. pushforward along proper morphism.
Note that any lci morphism can be factored as a composition of regular embedding followed by a projection map which is flat. The Functoriality of the cycle class map for pullbacks under the projection map is clear. See \cite[Chapter 19, Lemma 19.2(a)]{Fulton-intersection} for the functoriality of the cycle class maps for pullback via a regular embedding.
See \cite[Lemma 19.1.2]{Fulton-intersection} for functoriality of the cycle class map under pushforward along proper morphisms.
\end{proof}

\subsection{Closed conical subsets of cotangent bundle} \label{subsec:conic}
Let $X$ be an algebraic variety, then it's cotangent bundle $T^*X$ is again an algebraic variety equipped with a canonical 1-form $\omega$. With this choice $d\omega$ is a closed 2-form on $T^*X$ making it into a symplectic manifold.
A Zariski closed subset $C \subset T^*X$ is called a conical subset if it is stable under the obvious action of $\mathbb{G}_{m}$ on $T^*X$. The first and perhaps the example that is most pertinent to us is $\overline{T^*_SX}$, the closure in the Zariski topology of the conormal bundle of $X$ along a smooth locally closed subset in the Zariski topology $S \subset X$(See \cite[Prop. 8.4.1]{Kashiwara-Schapira}).
Let $C^{\mathrm{sm}} \subset C$ be the smooth locus of $C$ which is in fact dense in the Zariski topology. We call $C$ to be isotropic if $d\omega|_{C^{\text{sm}}} \equiv 0$, and $C$ is said to be involutive if for all $p \in C^{\mathrm{sm}}$, $T_pC$ is an involutive subset of $T_pT^*X$ for the symplectic pairing on $T_pT^*X$ induced by $\omega$. A subset $C$ is called lagrangian if $C$ is an isotropic and involutive subset.\\

\noindent
In the case of complex analytic variety $X^{an}$, the above paragraph must be read replacing Zariski topology, and $\mathbb{G}_{m}$ respectively by Euclidean topology and, $\mathbb{C}^{\times}$.

\section{Singular supports and characteristic cycles}

\subsection{Analytic} \label{subsec:analytic}
Let $\Lambda \in \lbrace \mathbb{Z}/\ell^n \mathbb{Z}, \mathbb{Z}_{\ell}, \mathbb{Q}_{\ell} \rbrace$ except in the last paragraph where we assume $\Lambda = \mathbb{Q}_{\ell}$.\\

\noindent
To any sheaf $F \in \mathrm{D}^b_c(X^{an}, \Lambda)$, Kashiwara-Schapira associates a closed conical isotropic involutive subset of $T^*X^{an}$ denoted in this article by $\mathrm{SS^{KS}}(F)$ (See \cite[Thm. 8.5.5]{Kashiwara-Schapira}).
The definition in the book perhaps cannot be seen to be immediately related to the definition of the singular supports due to Beilinson. We quote here a result from the book which resembles Beilinson's definition
\begin{theorem-definition}[\cite{Kashiwara-Schapira}, Prop. 8.6.4] \label{thm:singular-support-KS}
Let $\pi :T^*X^{an} \rightarrow X^{an}$ be the cotangent bundle of $X^{an}$. The following are equivalent
\begin{enumerate}
	\item $p \in T^*X^{an}$ does not belong to the singular support of $F$.
	\item There exists an open neighbourhood $U$\footnote{$U$ is open subset in the Euclidean topology.} of $p$ and a holomorphic function $f$ defined on some open neighbourhood $V$ of $\pi(p)$ satisfying $f(\pi(p)) = 0$ and $df(\pi(p)) \in U$, such that $\phi^{an}_f(F)_{\pi(p)} = 0$.
\end{enumerate}
\end{theorem-definition}

\noindent
Kashiwara-Schapira further define a cycle which is a formal sum of closed conical Lagrangian subsets with certain integer coefficients as explained below. Let $\pi : T^*X^{an} \rightarrow X^{an}$ be the cotangent bundle of $X^{an}$, $p_1,p_2$ respectively be the first and second projections of $X^{an} \times X^{an} \rightarrow X^{an}$, and $\delta : X^{an} \rightarrow X^{an} \times X^{an}$ be the diagonal embedding.
The \text{characteristic class} $\mathrm{C}(F) \in \mathrm{H}^0_{\mathrm{supp}(F)}(X^{an}, \omega_{X^{an}})$ is defined to be the image of the identity map of $F$, denoted by $\mathrm{id}_F \in \hom(F, F)$, under the following composite
\[
	\mathrm{R}\hom(F,F) \xrightarrow{\sim} \delta^!(F \boxtimes \mathrm{D}_{X^{an}}F) \rightarrow \delta^{*}(F \boxtimes D_{X^{an}}F) \xrightarrow{\sim} F \otimes D_{X^{an}}F \xrightarrow{tr} \omega_{X^{an}}.
\]
Here $\omega_X (=\Lambda_X[2\dim X])$ denote the dualizing sheaf of $X$. The above morphisms can be lifted to a map of sheaves on the cotangent bundle of $X^{an}$ using the technique of microlocalization. We refer the reader to \cite[\S 9.4]{Kashiwara-Schapira} for the definition of the characteristic cycle, and to \cite[Ch. IV]{Kashiwara-Schapira} for the definition of microlocalization. This allows us to write a morphism (See \cite[pp. 352]{Kashiwara-Schapira})
\begin{equation} \label{eqn:char-cycle-KS}
	\mathrm{R}\hom(F,F) \rightarrow R\pi_*R\Gamma_{\mathrm{SS}^{\mathrm{KS}}(F)}(\pi^{-1}\omega_{X^{an}}).
\end{equation}
The image of $\mathrm{id}_F \in \mathrm{R}\hom(F,F)$ under the above map is an equivalence class $\mathrm{H}^0_{\mathrm{SS^{KS}(F)}}(T^*X^{an}, \pi^{-1}\omega_{X^{an}})$, to be denoted by $\mathrm{CC^{KS}}(F)$.
We now make the assumption that $\Lambda$ is a field of characteristic 0.
Since $X^{an}$ is a complex manifold, we have a canonical isomorphism $\omega_{X^{an}} \simeq \Lambda_{X^{an}}[2\dim X^{an}]$. Denote by $\mathscr{CS}^{\bullet}(T^*X^{an})$ the sheaf of subanalytic chains in $T^*X^{an}$ (See \cite[\S 9.2]{Kashiwara-Schapira}). Using \cite[Prop. 9.2.6(iv)]{Kashiwara-Schapira} we get the following series of isomorphisms
\begin{equation} \label{eqn:lagrangian-cycle-cohomology-class}
	\begin{split}
	\mathrm{H}^0_{\mathrm{SS^{KS}(F)}}(T^*X^{an}, \Lambda_{T^*X^{an}}[2\dim X^{an}]) &= \mathrm{H}^{-2\dim X^{an}}_{\mathrm{SS^{KS}}(F)}(T^*X^{an}, \mathscr{CS}^{\bullet}(T^*X^{an})) = \mathscr{CS}^{T^*X^{an}}_{2\dim X^{an}}(\mathrm{SS^{KS}}(F))\\
								&\subset \Bigg\lbrace \sum a_iX^{an}_i \;\Bigg|\; \begin{aligned} &X^{an}_i \subset \mathrm{SS^{KS}}(F) \text{ locally closed subanalytic},\\ &\dim(X^{an}_i) = 2\dim X^{an} \text{ and }a_i \in \Lambda \end{aligned}\Bigg\rbrace.
	\end{split}
\end{equation}
The image of the cohomology class $\mathrm{CC^{KS}}(F)$ under the identification in \eqref{eqn:lagrangian-cycle-cohomology-class} will again be denoted by $\mathrm{CC^{KS}}(F)$. We will also need the following
\begin{lemma} \label{lemma:support-analytic-cycle}
	For a perverse sheaf $F$, $\mathrm{CC^{KS}}(F) \geq 0$ and is supported on $\mathrm{SS^KS}(F)$, the singular support of $F$.
\end{lemma}
\begin{proof}
The singular support commutes with the Riemann-Hilbert correspondence functor $\mathrm{DR}_{X^{an}}$ (See \cite[Theorem 11.3.3]{Kashiwara-Schapira}). The characteristic cycle commutes with $\mathrm{DR}_{X^{an}}$ as well (See \cite[pp.1115-1116]{Schmid-Vilonen}). But for holonomic $D$-modules the assertion that the characteristic cycle is supported on the characteristic variety is clear from the definition.
\end{proof}

\subsection{Algebraic} \label{subsec:algebraic}
We assume $\Lambda = \mathbb{Z}/\ell^n\mathbb{Z}$ in this section unless otherwise mentioned. In \cite{Beilinson-constructible} Beilinson defines the notion of the weak singular support of a constructible sheaf $F \in \mathrm{D}^b_c(X_{\text{\'{e}t}}, \Lambda)$ to be the smallest closed conical subsets $C$ of $T^*X$ satisfying the following : for every $C$-transversal\footnote{Recall that a pair $(h,f)$ is said to be $C$-\textit{transversal} if $df_x^{-1}(C_x) \setminus \lbrace 0 \rbrace = \emptyset$. Here $C_x$ denotes the set $C \cap T^*_xX$, and $df_x$ denotes the stalk at $x \in X$ of the morphism $df : T^*\mathbb{A}^1 \rightarrow T^*X$.}
test pair $(j,f)$ with $j : U \hookrightarrow X$ an open embedding, and a morphism $f : X \rightarrow \mathbb{A}^1$, $F|_U$ is locally acyclic with respect to the map $f$. Explicitly the weak singular support can also be described as the Zariski closure in $T^*X$ of the set
\[
	\lbrace (x, df(x)) \;|\; x \in X(\mathbb{C}) \text{ and } f \text{ is locally acyclic relative to $F$ at $x \in X(\mathbb{C})$} \rbrace.
\]
It is proved that $\mathrm{SS^B}(F)$ is a closed conical isotropic subset and each of it's irreducible components are of dimension $n = \dim (X)$ (See \cite[Prop. 2.2.7]{Saito-proper}).\\

\begin{remark}
\begin{enumerate}
	\item Let $f : U \rightarrow \mathbb{A}^1$ be a $\overline{\mathrm{SS^{KS}}(F)}^{Zar}$-transversal pair, then $df(x) \notin \overline{\mathrm{SS^{KS}}(F)}^{Zar}$ for any $x \in U$. Hence $\phi^{an}_f(F|_U) = 0$.
Using Theorem \ref{thm:comp-vanishing-cycle} we get that $\phi^{alg}_f(F|_U) = 0$. Thus, $F$ is microsupported on the Zariski closed subset $\overline{\mathrm{SS^{KS}}(F)}^{Zar}$. Hence the inclusion $\mathrm{SS^B}(F) \subset \overline{\mathrm{SS^{KS}}(F)}^{Zar}$ holds.
	\item For the sake of clarity we mention here that the notion of the weak singular support coincides with the notion of singular support as has been proved by Beilinson \cite[\S 1.5, Theorem]{Beilinson-constructible}.
\end{enumerate}
\end{remark}

\noindent
To any $F \in \mathrm{D}^b_{ctf}(X, \Lambda)$, Saito \cite{Saito} associates a cycle (not just a class!) supported on $\mathrm{SS^B}(F)$ with integer coefficients (See \cite[Prop. 5.18]{Saito}). We denote this cycle by $\mathrm{CC^S}(F)$.
More precisely if $\mathrm{SS^B}(F) = \underset{i}{\cup} C_i$ then $\mathrm{CC^S}(F) \colonequals \underset{i}{\sum} m_i[C_i]$ is such that for any triple $(j,f,u)$ in the set
\[
	\lbrace (	j: U \rightarrow X,f : U \rightarrow \mathbb{A}^1, u \in U(\mathbb{C})) \;|\; (j,f) \text{ is a test pair and}, (j|_{U \setminus u}, f|_{U\setminus u}) \text{ is a } C\text{-transversal test pair}. \rbrace,
\]
the equality
\begin{equation} \label{eqn:Milnor-formula}
	-\mathrm{tot}\dim(\phi^{alg}_u(j^*F, f)) = (\mathrm{CC^S}(F), df)_{T^*U, u} \tag{Milnor formula}
\end{equation}
holds. Here $\phi^{alg}_u(j^*F,f)$ denotes the vanishing cycle of $j^*F$ with respect to the morphism $f$ and $(\mathrm{CC^S}(F), df)_{T^*U, u}$ denotes the intersection of the cycle $\mathrm{CC^S}(F)$ with the graph of the map $df$ induced by the morphism $f$. The intersection number is well defined since $f$ is assumed to be transversal to $\mathrm{SS^B}(F)\setminus \lbrace u \rbrace$.
For a perverse sheaf $F$, the coefficients of the cycle $\mathrm{CC^S}(F)$ are nonnegative integers (\cite[Prop. 5.14]{Saito}).\\

\noindent
For the remaining part of this subsection we assume $\Lambda = \mathbb{Z}_{\ell}$. For any torsion free Zariski constructible \'{e}tale sheaf $\mathcal{F}$ with coefficients in $\Lambda$, there exists \'{e}tale sheaves $\mathcal{F}_n$ with coefficients in $\Lambda/\lambda^n$ such that $\mathcal{F}_n$ is flat over $\Lambda/\lambda^n$, and $\mathcal{F}_{n+1} \otimes_{\Lambda/\lambda^{n+1}} \Lambda/\lambda^n \cong \mathcal{F}_n$.
It follows from the definitions of the singular support and of the characteristic cycle that $\mathrm{SS^B}(\mathcal{F}_{n+1}) = \mathrm{SS^B}(\mathcal{F}_n)$ and $\mathrm{CC^S}(\mathcal{F}_{n+1}) = \mathrm{CC^S}(\mathcal{F}_n)$. Hence, it is meaningful to define
\[
	\mathrm{SS^B}(\lbrace \mathcal{F}_n \rbrace_n) \colonequals \mathrm{SS^B}(\mathcal{F}_0), \quad \mathrm{CC^S}(\lbrace \mathcal{F}_n \rbrace_n) \colonequals \mathrm{CC^S}(\mathcal{F}_0).
\]
It is proved in \cite[Prop. 5.9]{Umezaki-Yang-Zhao-20} that $\mathcal{F}$ is in fact mircosupported on $\mathrm{SS^B}(\mathcal{F})$ and that $\mathrm{CC^S}(\mathcal{F})$ satisfies the \eqref{eqn:Milnor-formula}.
For any $F \in \mathrm{D}^b_c(X_{\text{\'{e}t}}, \Lambda)[\lambda^{-1}]$, there exist torsion free sheaves $\mathcal{F}^i$ with coefficients in $\Lambda$ such that $\mathcal{F}^i\otimes_{\Lambda} \mathrm{Frac}(\Lambda) \cong \mathrm{H}^i(F)$.
Following a suggestion of Saito, Umezaki-Yang-Zhao \cite{Umezaki-Yang-Zhao-20} defines
\[
	\mathrm{CC^{SUYZ}}(F) \colonequals \sum_i (-1)^i \mathrm{CC^{S}}(\mathcal{F}^i), \text{ and} \quad	\mathrm{SS^{BUYZ}}(F) \colonequals \bigcup_i\mathrm{Supp}\left(\mathrm{CC^{SUYZ}}({}^p\mathrm{H}^i(F))\right).
\]
In the following proposition we list and indicate a quick proof of some expected properties of $\mathrm{SS^{BUYZ}}$ and $\mathrm{CC^{SUYZ}}$.
\begin{prop} \label{prop:basic-properties-CC-SS-SUYZ}
With the notation as above, the following holds
\begin{enumerate}
	\item For any $F \in \mathrm{D}^b_c(X_{\textup{\'{e}t}}, \mathrm{Frac}(\Lambda))$, $\mathrm{SS^{BUYZ}}(F) = \bigcup_i \mathrm{SS^{BUYZ}}({}^p\mathrm{H}^i(F))$ and $\mathrm{CC^{SUYZ}}(F) = \sum_i \mathrm{CC^{SUYZ}}({}^p\mathrm{H}^i(F))$.
	\item For a perverse sheaf $F \in \mathrm{D}^b_c(X_{\textup{\'{e}t}}, \Lambda)[\lambda^{-1}]$, the coefficients of the cycle $\mathrm{CC^{SUYZ}}(F)$ are all nonnegative.
	\item Let $0 \rightarrow F_1 \rightarrow F_2 \rightarrow F_3 \rightarrow 0$ be an exact sequence of perverse sheaves in $\mathrm{D}^b_c(X_{\textup{\'{e}t}}, \Lambda)[\lambda^{-1}]$. Then $\mathrm{SS^{BUYZ}}(F_2) = \mathrm{SS^{BUYZ}}(F_1) \cup \mathrm{SS^{BUYZ}}(F_3)$ and $\mathrm{CC^{SUYZ}}(F_2) = \mathrm{CC^{SUYZ}}(F_1) + \mathrm{CC^{SUYZ}}(F_3)$.
	\item Suppose that $U \overset{j}{\hookrightarrow} X$ is an open subset of $X$ such that $Z = X \setminus U = \underset{i=1}{\overset{n}{\cup}} D_i$ is a normal crossing divisor. Set $D_I \colonequals \cap_{\lbrace i \in I \rbrace} D_i$ and $D_{\emptyset} \colonequals X$.
		Let $F$ be a locally constant constructible sheaf on $U$ valued in $\mathrm{Frac}(\Lambda)$ whose complement is $D$. Then $\mathrm{SS^{BUYZ}}(j_!F[\dim X]) = \underset{I \subset \lbrace 1, 2, \dots, n \rbrace}{\cup} T^*_{D_I}X$ and $\mathrm{CC^{SUYZ}}(j_!F[\dim X]) = \rank(F)\underset{I \subset \lbrace 1, 2, \dots, n \rbrace}{\sum}T^*_{D_I}X$.
\end{enumerate}
\end{prop}
\begin{proof}
\begin{enumerate}
	\item The equality on the singular supports and the characteristic cycles is clear from the definition.
	\item We know from \cite[Thm. 5.17(3)]{Umezaki-Yang-Zhao-20} that $\mathrm{CC^{SUYZ}}$ satisfies the Milnor number formula. We may use \cite[\S 4.9(i)]{Beilinson-constructible}, more precisely it's refinement \cite[Lemma 4.10]{Saito} to conclude that the coefficients are nonnegative.
	\item The equality on the characteristic cycles follows from the definition. The equality on the singular supports follows from the definition and part (2).
	\item Note that $j_!$ is $t$-exact and hence, $j_!F[\dim X]$ is also perverse. In this case the equality $\mathrm{SS^{BUYZ}}(j_!F) = \mathrm{Supp}(\mathrm{CC^{SUYZ}}(j_!F))$ follows from the definitions of $\mathrm{SS^{BUYZ}}$ and $\mathrm{CC^{BUYZ}}$.
		Thus it is enough to prove that $\mathrm{CC^{SUYZ}}(j_!F[\dim X]) = \rank(F)\underset{I \subset \lbrace 1, 2, \dots, n \rbrace}{\sum}T^*_{D_I}X$. This follows by putting together \cite[Thm. 5.17(1)]{Umezaki-Yang-Zhao-20} and \cite[Prop. 4.11]{Saito}.
\end{enumerate}
\end{proof}

\subsection{Three lemmas} \label{subsec:lemmas}
In this subsection we assume that $\Lambda = \mathbb{Q}_{\ell}$.\\

\noindent
Consider the following diagram $\begin{tikzcd} U \arrow[r, hook, "j"] & X & D \arrow[l, hook', "i"'] \end{tikzcd}$ where $j$ is an open immersion and $i$ is a closed immersion such that $D = \underset{i=1}{\overset{r}{\cup}} D_i$ is a strict normal crossing divisor.
Let $F$ be a local system on the open set $U^{an}$, then $j_{!}F$ is a perverse sheaf on $X^{an}$ (See \cite[Cor. 4.1.10]{BBD}).
\begin{lemma} \label{lemma:char-cycle-shriek}
Let $X, D$, and $U$ be as above. Define $D_I \colonequals \cap_{\lbrace i \in I \rbrace} D_i$ and $D_{\emptyset} \colonequals X$. Then, for any $F \in \mathrm{D}^b_c(X^{an}, \Lambda)$, we have
\[
	\mathrm{CC^{KS}}(j_!F) = (-1)^{\dim_{\mathbb{C}} X^{an}} \rank(F|_U) \underset{I\subset \lbrace 1, 2, \dots, r \rbrace}{\sum}T^*_{D^{an}_I}X^{an}.
\]
\end{lemma}
\begin{proof}
Since the question is local (in fact microlocal) we may assume that $X^{an} = \mathbb{C}^n$, $D = \lbrace z_1z_2\dots z_r = 0 \rbrace$, and $F = j_!L$, where $L$ is a $\mathbb{Q}_{\ell}$-local system on $(\mathbb{C}^{\times})^r\times \mathbb{C}^{n-r}$.
It is clear from the definition of $\mathrm{CC^{KS}}$ that $\mathrm{CC^{KS}}(F) = \mathrm{CC^{KS}}(F \otimes_{\Lambda} \mathbb{C})$. Thus we may further assume $F = j_!(L\otimes \mathbb{C})\in \mathrm{D}^b_c(X^{an}, \mathbb{C})$.
Since $\mathrm{CC^{KS}}$ is additive under triangles (See \S\ref{subsec:summary}(2)), in addition we may assume that $L$ is an irreducible local system with complex coefficients. So we may write $L = L_1 \boxtimes \dots \boxtimes L_r$\footnote{Note that all the irreducible representations of $\pi_1((\mathbb{C}^{\times})^r) = \mathbb{Z}^r$ over an algebraically closed field is one dimensional and is product of one dimensional representation.}
for certain irreducible local system $L_i$ of $\mathbb{C}^{\times}$. Let $j_r : \mathbb{C}^{\times} \hookrightarrow \mathbb{C}$ be the restriction of $j$ to the $r$-th coordinate, then $j^* = j^*_1\times \dots \times j^*_r$.
Since $j_!$ is left adjoint to $j^* = j^*_1\times \dots \times j^*_r$, $j_!L = j_{1!}L_1\boxtimes j_{2!}L_2 \boxtimes \cdots \boxtimes j_{r!}L_r$. We know from \cite[pp. 378]{Kashiwara-Schapira} that $\mathrm{CC^{KS}}(j_!L) = \mathrm{CC^{KS}}(j_{1!}L_1)\boxtimes \cdots \boxtimes \mathrm{CC^{KS}}(j_{r!}L_r)$.
Supposing we also know that $\mathrm{CC^{KS}}(j_{t!}L_t) = -\rank(L_t)([T^*_{\mathbb{C}}\mathbb{C}] + [T^*_{\lbrace 0 \rbrace} \mathbb{C}])$, we get
\[
	\mathrm{CC^{KS}}(j_!L) = (-1)^{\dim X}\rank(L)([T^*_{\mathbb{C}}\mathbb{C}] + [T^*_{\lbrace 0 \rbrace} \mathbb{C}]) \boxtimes \cdots \boxtimes([T^*_{\mathbb{C}}\mathbb{C}] + [T^*_{\lbrace 0 \rbrace} \mathbb{C}]) = (-1)^{\dim X}\rank(L) \underset{I\subset \lbrace 1, \dots, r \rbrace}{\sum} T^*_{D_I}X.
\]
It remains to prove the equality $\mathrm{CC^{KS}}(j_{!}L) = -\rank(L)([T^*_{\mathbb{C}}\mathbb{C}] + [T^*_{\lbrace 0 \rbrace} \mathbb{C}])$ where $L$ is a one dimensional local system with coefficients in $\mathbb{C}$.
We first prove this when $L$ is a trivial local system. Applying 1 of \S \ref{subsec:summary} to the triangle $j_{!}j^*\mathbb{C} \rightarrow \mathbb{C} \rightarrow i_{*}i^*\mathbb{C}$ we get
\[
	\mathrm{CC}^{KS}(j_{!}\mathbb{C}) = \mathrm{CC^{KS}}(\mathbb{C}) - \mathrm{CC^{KS}}(i_*\mathbb{C}) = -[T^*_{\mathbb{C}}\mathbb{C}] - [T^*_{\lbrace 0 \rbrace} \mathbb{C}].
\]
Thus proving the claim in this case. If $L$ is a nontrivial local system then the canonical map $j_!L \rightarrow j_*L$ is an isomorphism. In this case we have the following commutative diagram
\[
\begin{tikzcd}
	& j_!(L \otimes L^{\vee}) \arrow[r, "\sim"] & j_!L\otimes j_*L^{\vee} \arrow[r, "\sim"] & j_!L \otimes D(j_!L) \arrow[dr]\\
	j_!\mathbb{C} \arrow[ur, "\sim"] \arrow[dr, "\sim"] &&&& \omega_{X^{an}}\\
	& j_!(\mathbb{C} \otimes \mathbb{C}) \arrow[r, "\sim"] & j_!(\mathbb{C}) \otimes j_*(\mathbb{C}) \arrow[r, "\sim"] & j_!(\mathbb{C}) \otimes D(j_!\mathbb{C}) \arrow[ur]
\end{tikzcd}
\]
Chasing the image of the identity morphism in $\mathrm{R}\hom(j_!L, j_!L)$ and $\mathrm{R}\hom(j_!\mathbb{C}, j_!\mathbb{C})$ in the sequence of arrows in the definition of characteristic cycle (See \cite[\S 9.4]{Kashiwara-Schapira}) we conclude using the diagram above that their images in $R\pi_*R\Gamma_{\mathrm{SS}}(\pi^{-1}\omega_{X^{an}})$ coincides. Hence $\mathrm{CC^{KS}}(j_!L) = \mathrm{CC^{KS}}(j_!\mathbb{C}) = -[T^*_{\mathbb{C}}\mathbb{C}]-[T^*_{\lbrace 0 \rbrace}\mathbb{C}]$.
This finishes the proof of the lemma.
\end{proof}

\begin{lemma} \label{lemma:char-cycle-equal-pushforward}
Let $p : Y \rightarrow X$ be a projective morphism of algebraic varieties and $F \in \mathrm{D}^b_c(Y_{\textup{\'{e}t}}, \Lambda)$. Assume $\Lambda \in \lbrace \mathbb{Z}/\ell^n \mathbb{Z}, \mathbb{Z}_{\ell} \rbrace$ and $\mathrm{CC^{KS}}(F) = \mathrm{CC^{S}}(F)$, or $\Lambda = \mathbb{Q}_{\ell}$ and $\mathrm{CC^{KS}}(F) = \mathrm{CC^{SUYZ}}(F)$. Then $\mathrm{CC^{KS}}(p_*F) = \mathrm{CC^{SUYZ}}(p_*F)$.
\end{lemma}
\begin{proof}
Let the pushforward in cycles (resp. cohomology) along the correspondence
	\[
	\begin{tikzcd}
	& T^*X\times_XY \arrow[dl, "dp"'] \arrow[dr]&\\
	T^*Y & & T^*X.
	\end{tikzcd}
	\]
be denoted by $p_!$ (resp. $p_*$). We have the following equality of cycles
	\[
		\mathrm{CC^{SUYZ}}(p_*F) \overset{(1)}{=} p_!\mathrm{CC^{SUYZ}}(F) \overset{(2)}{=} p_*\mathrm{CC^{KS}}(F) \overset{(3)}{=} \mathrm{CC^{KS}}(p_*F).
	\]
Using \cite[Thm. 5.17(1)]{Umezaki-Yang-Zhao-20} we are reduced to proving (1) for a Zariski constructible \'{e}tale sheaf $F$ with coefficients in $\mathcal{O}_{\Lambda}/\lambda^n$.
The equality (1) for $F \in \mathrm{D}^b_c(Y_{\text{\'{e}t}}, \mathcal{O}_{\Lambda}/\lambda^n)$ is the content of \cite[Prop. 2.2.7(2)]{Saito-proper}.
Saito first proves that the equality holds under certain addtional assumption on the dimension of $f_0(\mathrm{SS}(F))$ (See \cite[Thm. 2.2.5]{Saito-proper}). He then proves that this additional assumption is automatically satisfied when $X$ is defined over a field of characteristic 0.
This finishes the proof of equality (1) above. The equality (3) is the content of \cite[Prop. 9.4.2]{Kashiwara-Schapira}. The equality (2) is explained in Lemma \ref{lemma:cycle-class-functorial} of \S\ref{subsec:cycle-class-functorial}.
\end{proof}

\begin{remark}
The assumption in the above lemma can be relaxed to - $f$ a quasi-projective map and proper on the support of $F$. This assumption is forced on us since we rely crucially on Saito's result \cite[Prop. 2.2.7(2)]{Saito-proper}.
\end{remark}

\begin{lemma} \label{lemma:irreducible-jordan-quotient}
Let $F$ be a perverse sheaf on $X$ such that $F|_U$ is isomorphic to $L|_U$ where $L|_U$ is a simple local system. Then $\mathrm{IC}_X(L|_U)$ is the only simple subquotient of $F$ with support containing the open subset $U$.
\end{lemma}
\begin{proof}
This statement is about composition series and hence we may assume that $F$ is semisimple. Moreover $F$ may be assumed to be simple since the local system of interest is simple. So it suffices observe that for a simple perverse sheaf $F$, $F|_U$ is isomorphic to a simple local system $L|_U$, then $F \cong \mathrm{IC}_X(L|_U)$.
\end{proof}

\subsection{Summary of properties of singular supports and characteristic cycles} \label{subsec:summary}
 The properties of the singular supports and the characteristic cycles are summarised below. In this subsection the notation $\mathrm{SS}(F)$ and $\mathrm{CC}(F)$ are used to signify that the properties enumerated below continues to hold true for both the notions of the singular support and of the characteristic cycle discussed in \S \ref{subsec:analytic} and \S \ref{subsec:algebraic}.
\begin{enumerate}
	\item $\mathrm{SS}(F) = \cup_i \mathrm{SS}({}^p\mathrm{H}^i(F))$, $\mathrm{CC}(F) = \underset{i}{\sum}\mathrm{CC}({}^p\mathrm{H}^i(F))$.  See \cite[Prop. 5.1.3(iii), Prop. 9.4.5]{Kashiwara-Schapira}, \cite[Lemma 5.13(1)]{Saito}, and Propisition \ref{prop:basic-properties-CC-SS-SUYZ}(1) above.
	\item For an exact sequence of perverse sheaves $0 \rightarrow F_1 \rightarrow F_2 \rightarrow F_3 \rightarrow 0$, the equalities $\mathrm{SS}(F_2) = \mathrm{SS}(F_1) \cup \mathrm{SS}(F_2)$ and $\mathrm{CC}(F_2) = \mathrm{CC}(F_1)+\mathrm{CC}(F_3)$ holds. See \cite[Prop. 9.4.5(ii)]{Kashiwara-Schapira}, \cite[Lemma 5.13(1)]{Saito}, and Proposition \ref{prop:basic-properties-CC-SS-SUYZ}(3).
	\item Suppose that $U \overset{j}{\hookrightarrow} X$ is an open subset of $X$ such that $Z = X \setminus U = \underset{i=1}{\overset{n}{\cup}} D_i$ is a normal crossing divisor. Let $F$ be a locally constant constructible sheaf on $U$, whose complement is $D$. Then $\mathrm{SS}(j_!F) = \underset{I \subset \lbrace 1, 2, \dots, n \rbrace}{\cup} T^*_{D_I}X$, and $\mathrm{CC}(j_!F) = (-1)^{\dim X} \rank(F)\underset{I \subset \lbrace 1, 2, \dots, n \rbrace}{\sum}T^*_{D_I}X$. See Proposition \ref{prop:basic-properties-CC-SS-SUYZ}(4) and Lemma \ref{lemma:char-cycle-shriek} above.
	\item Let $p : X \rightarrow Y$ be a projective morphism between smooth varieties. Then $\mathrm{CC^{KS}}(F) = \mathrm{CC^{SUYZ}}(F) \Rightarrow \mathrm{CC^{KS}}(p_*F) = \mathrm{CC^{SUYZ}}(p_*F)$. See Lemma \ref{lemma:char-cycle-equal-pushforward} above.
\end{enumerate}

\section{Main theorem and proof} \label{sec:main-theorem-proof}
\begin{theorem} \label{thm:main}
Let $X$ be a smooth algebraic variety, and $F \in \mathrm{D}^b_c(X_{\textup{\'{e}t}}, \mathbb{Q}_{\ell})$. Then $\mathrm{CC^{SUYZ}}(F) = \mathrm{CC^{KS}}(F)$.
\end{theorem}

\begin{proof}
	We may assume $F \neq 0$ since the theorem is clear when $F=0$. The proof of the theorem proceeds via an induction argument on the dimension of the support. The case of objects $F \in \mathrm{D}^b_c(X_{\text{\'{e}t}}, \mathbb{Q}_{\ell})$ supported of dimension 0 is obvious. We may thus begin the induction process.
\begin{enumerate}
	\item[Step 1.] \textit{Reduce to $F$ a simple perverse sheaf}. An application of \S\ref{subsec:summary}(1) implies that it is enough to prove the theorem for a perverse sheaf $F$. Using \S \ref{subsec:summary}(2) we are reduced to considering only simple perverse sheaves.\\

	\item[Step 2.] \textit{Behaviour under taking subobjects and quotients}. Let $0 \rightarrow F_1 \rightarrow F_2 \rightarrow F_3 \rightarrow 0$ be an exact sequence of perverse sheaves. If the theorem holds for any two out of the three perverse sheaves $F_1,F_2$, and $F_3$, then it holds for the remaining one as well.
This is again clear by using \S \ref{subsec:summary}(2).\\

	\item[Step 3.] \textit{Reduction to $F=j_{!*}L_U$ where $j : U \hookrightarrow X$ is an open embedding whose complement is strict normal crossing divisor}. Let $F$ be a simple perverse sheaf. We may assume that $F = j_{!*}L_U$, where $L_U$ is a simple local system on $U$.
After possibly choosing a smaller open subset $U$ and using Hironaka's theorem on resolution of singularities, we can ensure that there exists a smooth variety $\widetilde{X}$ and a projective birational map $r : \widetilde{X} \rightarrow X$,
such that $r|_{r^{-1}(U)} : r^{-1}(U) \rightarrow U$ is an isomorphism and $\begin{tikzcd} r^{-1}(U) \arrow[r, hook, "\tilde{j}"] & \widetilde{X} & \arrow[l, hook', "\tilde{i}"'] Z \end{tikzcd}$, where $Z \subset \widetilde{X}$ is a strict normal crossing divisor.
Let $\widetilde{F} \colonequals j_{!*}L_{r^{-1}(U)}$. Then using the decomposition theorem with respect to supports for perverse sheaf $\widetilde{F}$, we get
\[
	r_*\widetilde{F} = {}^{p}\mathrm{H}^0(r_*F) \oplus \lbrace \text{shifted direct sum of perverse sheaves supported on smaller dimensional subvarieties} \rbrace.
\]
Assuming that the theorem holds for $\widetilde{F}$, we use \S \ref{subsec:summary}(4) to conclude that the theorem holds for $r_*\widetilde{F}$. An application of Step 2 above, and the hypothesis that the theorem holds for all perverse sheaves supported on closed subsets of $X$ of dimension $<\dim X$ implies that the theorem holds for ${}^{p}\mathrm{H}^0(r_*F)$.
Using Lemma 2.5 we get that $F \hookrightarrow {}^p\mathrm{H}^0(r_*\widetilde{F})$ with the quotient being supported on smaller dimensional smooth varieties. An application of step 2 and the hypothesis that the theorem holds for all perverse sheaves supported on closed subsets of $X$ of dimension $<\dim X$ we get that the theorem is true for $F$. Thus completing the proof of this step.\\

	\item[Step 4.] \textit{Proof of the theorem for $F$ as in Step 3}. Let $F$ be a perverse sheaf on a smooth variety such that $F = j_{!*}(L_U)$ and $X \setminus U$ is a strict normal crossing divisor. The triangle $\cdots \rightarrow j_!j^*F \rightarrow F \rightarrow i_*i^*F \rightarrow \cdots$ gives the following exact sequence of perverse sheaves (See \cite[Cor. 4.1.10(ii)]{BBD})
	\[
		0 \rightarrow i_*{}^p\mathrm{H}^{-1}i^*F \rightarrow j_!j^*F \rightarrow F \rightarrow i_*{}^p\mathrm{H}^0i^*F \rightarrow 0.
	\]
To prove the theorem for $F$, it is enough to prove the theorem for all objects in the the above exact sequence other than $F$. By inductive hypothesis we may asusme that the theorem holds for the extreme terms.
Thus it is enough to prove the theorem for $j_!j^*F$ which follows form \S \ref{subsec:summary}(3).
\end{enumerate}
\end{proof}

\begin{corollary}
With notation as in the previous theorem, $\mathrm{SS^{BUYZ}}(F) = \mathrm{SS^{KS}}(F)$
\end{corollary}
\begin{proof}
Let $F$ be a perverse sheaf. We know from \cite[Prop. 5.14(2)]{Saito-proper} that support of the cycle $\mathrm{CC^S}(F)$ is $\mathrm{SS^B}(F)$. It is clear from the definitions of the extended notions of characteristic cycles to $\mathbb{Q}_{\ell}$-sheaves that the support of the cycle $\mathrm{CC^{SUYZ}}(F)$ is $\mathrm{SS^{BUYZ}}(F)$.
On the other hand, the arguments of Kashiwara-Schapira and Schmid-Vilonen as summarized in Lemma \ref{lemma:support-analytic-cycle} implies that support of the Lagrangian cycle $\mathrm{CC^{KS}}(F)$ is $\mathrm{SS^{KS}}(F)$.
Now taking supports of both sides of the equality established in Theorem \ref{thm:main} we get that for a perverse sheaf $F$, the equality $\mathrm{SS^{BUYZ}}(F) = \mathrm{SS^{KS}}(F)$ holds.
Now using \S \ref{subsec:summary}(1) the equality holds for all $F \in \mathrm{D}^b_c(X_{\text{\'{e}t}}, \mathbb{Q}_{\ell})$.
\end{proof}

\appendix
\section{Coefficients in $\mathrm{CC^{KS}}$ and $\mathrm{CC^S}$} \label{appendix:sec-coefficient-system}
The characteristic cycle $\mathrm{CC^S}(F)$ of a constructible sheaf $F \in \mathrm{D}^b(X, \Lambda)$ defined by Saito have coefficients in the ring of integers irrespective of the ring of coefficients $\Lambda$. On the other hand characteristic cycle as constructed\footnote{Note that in \cite[\S 9.4]{Kashiwara-Schapira}, it is assumed that $\Lambda = k$ is a field of characteristic 0, but clearly the maps make sense for any ring $\Lambda$ of finite global dimension.} by Kashwara-Schapira have coefficients in the ring $\Lambda$.
Under the assumption that $\Lambda$ is a field of characteristic 0, it is proved (\cite[Prop. 9.4.5]{Kashiwara-Schapira}) that characteristic cycle $\mathrm{CC^{KS}}(-)$ has coefficients in the ring $\mathbb{Z}$.
In this section we wish to understand the dependence of the characteristic cycle on the coefficient system.\\

\noindent
Let $\Lambda \in \lbrace \mathbb{Z}/\ell^n\mathbb{Z}, \mathbb{Z}_{\ell} \rbrace$ and $F \in \mathrm{D}^b_c(X_{\text{\'{e}t}}, \Lambda)$. For convenience we denote the sheaf $\varepsilon^*F$ again by $F$ in this section. Recall that $\mathrm{CC^S}(F)$ is an element of $\mathrm{CH}^{2\dim X}(\mathrm{SS^B}(F))$ and $\mathrm{CC^{KS}}(F)$ is an element of $\mathrm{H}^{0}_{\mathrm{SS}(F)}(T^*X^{an}, \omega_{X^{an}})$.
Since $\Lambda$ is a noetherian ring we may associate a well-defined integer $\dim \mathrm{tot}(K)$ to a perfect complex $K \in \mathrm{D}^b_{perf}(\Lambda\text{-mod})$.
We proceed as in \cite[pp. 382]{Kashiwara-Schapira};  for $F \in \mathrm{D}^b_c(X^{an}, \Lambda)$ and $p$ in an irreducible component $V$ of $\mathrm{SS}(F)$ there exists $K \in \mathrm{D}^b(\Lambda\text{-mod})$ such that $F \xrightarrow{\sim} A$ (\cite[Prop. 6.6.2]{Kashiwara-Schapira}) in $\mathrm{D}^b(X^{an};p)$ (\cite[\S 6.1]{Kashiwara-Schapira}) and we define $m_V \colonequals \dim \mathrm{tot}(A)$.
Then the image of $id \in \mathrm{R}\underline{\hom}(F, F)_p = \mathrm{H}^0(\mu\hom(F, F))_p$ (\cite[Thm. 6.1.2]{Kashiwara-Schapira}) in $\Lambda$ is given by the image of $\dim\mathrm{tot}(A)$ in the ring $\Lambda$ under the canonical map $\mathbb{Z} \rightarrow \Lambda$. Now using Theorem \ref{thm:main} we get

\begin{lemma}
The image of $\mathrm{CC^{S}}(F)$ under the canonical map $\mathbb{Z} \rightarrow \Lambda$ is the cycle $\mathrm{CC^{KS}}(F)$.
\end{lemma}

\section{Characteristic cycles for $\mathbb{Q}_{\ell}$-sheaves} \label{appendix-B}
We have already seen that the notion of $\mathrm{CC^S}$ has been extended for $F \in \mathrm{D}^b_c(X_{\text{\'{e}t}}, \mathbb{Q}_{\ell})$ by Umezaki-Yang-Zhao.
In a recent article \cite{Barrett-constructible}, Barrett has extended the definition of singular support $\mathrm{SS^B}$ by utilizing the interpretation of the condition of local acylicity ($\equiv$ universal local acyclicity) in terms of dualizable objects in a certain 2-monoidal category developed by Lu-Zheng and Hansen-Scholze.
This allows him to bypass nonfinite type schemes such as Milnor fibers (denoted $M_{\bar{x}, \bar{s}}$ in \S \ref{subsec:vanishing-cycles}) over which six functors maynot preserve the derived category of Zariski constructible \'{e}tale $\mathbb{Q}_{\ell}$-sheaves.
He also uses the pro\'{e}tale topology to bypass the 2-limit construction of derived category of $\mathbb{Q}_{\ell}$-sheaves which allows for cleaner arguments once certain technical results are proved (See \cite[\S 2, \S 3]{Barrett-constructible}).\\

\noindent
In this article Barrett constructs a torsion free $\mathbb{Z}_{\ell}$-model of a $\mathbb{Q}_{\ell}$-perverse sheaves. More precisely it is shown that given a perverse sheaf $F \in \mathrm{Perv}(X, \mathbb{Q}_{\ell})$, there exists $\mathcal{F} \in \mathrm{Perv}(X, \mathbb{Z}_{\ell})$ which is torsion free and $\mathcal{F}\otimes \mathbb{Q}_{\ell} \cong F$.
The singular support for $F$ can now be defined as
\[
	\mathrm{SS^{BB}}(F) \colonequals \mathrm{SS^B}(\mathcal{F}\otimes \mathbb{Z}_{\ell} /\ell), \text{ where $\mathcal{F}$ is torsion free and } \mathcal{F}\otimes \mathbb{Q}_{\ell} \cong F.
\]
It is natural to extend the definition of characteristic cycles along the same lines and define
\[
	\mathrm{CC^{SB}}(F) \colonequals \mathrm{CC^S}(\mathcal{F} \otimes \mathbb{Z}_{\ell}/\ell).
\]
It can be proved as in \cite[Prop. 5.9]{Umezaki-Yang-Zhao-20}, that $\mathrm{CC^{SB}}$ satisfies the Milnor number formula. Thus we may prove the following
\begin{lemma}
For a perverse sheaf $F \in \mathrm{Perv}(X_{\textup{\'{e}t}}, \mathbb{Q}_{\ell})$, we have $\mathrm{CC^{SUYZ}}(F) = \mathrm{CC^{SB}}(F)$.
\end{lemma}
\begin{proof}
This follows immediately from \cite[Lemma 4.10]{Saito} and the fact that both the notions satisfy Milnor number formula.
\end{proof}
\noindent
We know from \cite[Prop. 5.14(2)]{Saito} that $\mathrm{CC^{SB}}(F) = \mathrm{SS^{BB}}(F)$ for any perverse sheaf $F \in \mathrm{Perv}(X_{\text{\'{e}t}}, \mathbb{Q}_{\ell})$. In light of the above lemma, we get $\mathrm{SS^{BB}} = \mathrm{SS^{BUYZ}}$.
Thus the definition of singular support $\mathrm{SS^{BUYZ}}$ coincides with the definition of singular supports developed by Barrett.

\AtEndDocument


\begin{thebibliography}{9999999}
\bibitem[Bar23]{Barrett-constructible} O. Barrett, {\em The singular support of an $\ell$-adic sheaf}, \url{http://arxiv.org/abs/2309.02587v1}.

\bibitem[Bei17]{Beilinson-constructible} A. Beilinson, {\em Constructible sheaves are holonomic}, Selecta Math. New Ser. 22 (2016), 1797--1819.

\bibitem[BBD82]{BBD} A. Beilinson, I. Bernstein, and P. Deligne, {\em Faisceaux perverse}, Asterisque 100 (1982), 5--171.

\bibitem[SGA$7_{\textup{II}}$]{Deligne-Katz} {\em Groupes de monodromie en g\'{e}om\'{e}trie alg\'{e}brique. {II}}, S\'{e}minaire de G\'{e}om\'{e}trie Alg\'{e}brique du Bois-Marie 1967--1969 (SGA $7_{\textup{II}}$), Lect. Notes in Math., Dirig\'{e} par P. Deligne et N. Katz, Springer-Verlag, Berlin-New York (1973).

\bibitem[El78]{ElZein} F. El Zein, {\em Complexe dualisant et applications a la classe fondamentale d'un cycle}, Bull. de la S.M.F. (1978), 5--93.

\bibitem[Ful]{Fulton-intersection} W. Fulton, {\em Intersection theory}, Ergebnisse der Mathematik und ihrer Grenzgebiete. 3. Folge. A Series of Modern Surveys in Mathematics, 2nd edition, Springer-Verlag, Berlin (1998).

\bibitem[HS23]{Hansen-Scholze} P. Scholze, and D. Hansen, {\em Relative perversity}, Comm. Amer. Soc., 3 (2023), 631--668.

\bibitem[KS]{Kashiwara-Schapira} M. Kashiwara, and P. Schapira, {\em Sheaves on manifolds}, Grundlehren der mathematischen Wissenschaften, Springer-Verlag, Berlin (2006).

\bibitem[LZ19]{Lu-Zheng} Q. Lu, and W. Zheng, {\em Duality and nearby cycles over a general base}, Duke Math. J. (2019), 3135--3213.

\bibitem[Org06]{Orgogozo-nearby-cycles} F. Orgogozo, {\em Modifications et cycles proches sur une base g\'{e}n\'{e}rale}, Int. Math. Res. Not. (2006).

\bibitem[Sai17]{Saito} T. Saito, {\em The characteristic cycle and the singular support of a constructible sheaf}, Invent. Math. (2017), 597--695.

\bibitem[Sai20]{Saito-proper} T. Saito, {\em Characteristic cycles and the conductor of direct image}, J. Amer. Math. Soc. (2020), 369--410.

\bibitem[SGA4$\frac{1}{2}$]{Deligne-SGA4.5-finiteness} P. Deligne, {\em Th\'{e}or\`{e}mes de finitude en cohomologie $\ell$-adique}, SGA 4$\frac{1}{2}$.

\bibitem[SV00]{Schmid-Vilonen} W. Schmid, K. Vilonen, {\em Characteristic cycles and wave front cycles of representations of reductive Lie groups}, Ann. Math., 2nd series (2000), 1071--1118.

\bibitem[Sch]{Schurman-constructible} J. Sch\"{u}rman, {\em Topology of singular spaces and constructible sheaves}, Instytut Matematyczny Polskiej Akademii Nauk. Monografie Matematyczne (New Series), Birkh\"{a}user Verlag, Basel (2003).

\bibitem[Sta]{stacks-project} The Stacks project authors, {\em The Stacks project}, \url{https://stacks.math.columbia.edu}, (2023).

\bibitem[UYZ20]{Umezaki-Yang-Zhao-20} N. Umezaki, E. Yang, and Y. Zhao, {\em Characteristics class and the $\varepsilon$-factor of an \'{e}tale sheaf}, Tran. Amer. Math. Soc. (2020), 6887--6927.

\end{thebibliography}
\end{document}